\documentclass[11pt]{article}
\usepackage{fullpage}

\usepackage{amsthm}
\usepackage{amssymb}
\usepackage{mathabx}
\usepackage{tikz,pgfplots}
\usepackage{graphicx}
\usepackage{subcaption}
\usepackage{float}

\theoremstyle{plain}
\newtheorem{theorem}{Theorem}[section]

\newtheorem{proposition}[theorem]{Proposition}
\newtheorem{subproposition}[theorem]{}

\newtheorem{corollary}[theorem]{Corollary}

\newenvironment{subproof}[1][\proofname]{%
  \begin{proof}[#1]%
}{%
  \end{proof}%
}
\theoremstyle{definition}
\newtheorem{definition}[theorem]{Definition}

\newcommand{\ba}{\backslash}
\newcommand{\closedN}[1]{\ensuremath{\widebar{\mathrm{Neigh}}(#1)}}
\newcommand{\openN}[1]{\ensuremath{\mathrm{Neigh}(#1)}}

\usepackage{authblk}

\title{A splitter theorem for connected clutters}

\author[a]{Amanda Cameron\footnote{Corresponding author. Email: amanda.cameron@qmul.ac.uk}}
\author[b]{Dillon Mayhew\thanks{Email: dillon.mayhew@vuw.ac.nz}}
\affil[a]{School of Mathematical Sciences, Queen Mary University of London, Mile End Road, London, United Kingdom E1 4NS}
\affil[b]{School of Mathematics, Statistics and Operations Research, Victoria University of Wellington, PO Box 600, Wellington, New Zealand}

\begin{document}
\maketitle
\begin{abstract}
 A clutter consists of a finite set and a collection of pairwise incomparable subsets. Clutters are natural generalisations of matroids, and they have similar operations of deletion and contraction. We introduce a notion of connectivity for clutters that generalises that of connectivity for matroids. We prove a splitter theorem for connected clutters that has the splitter theorem for connected matroids as a special case: if $M$ and $N$ are connected clutters, and $N$ is a proper minor of $M$, then there is an element in $E(M)$ that can be deleted or contracted to produce a connected clutter with $N$ as a minor.
\end{abstract}

\section{Introduction}

A \emph{clutter} is a pair $(E,\mathcal{A})$, where $E$ is a finite set, and $\mathcal{A}$ is a collection of subsets of $E$, with the property that if $A$ and $A'$ are distinct members of $\mathcal{A}$, then $A\nsubseteq A'$. We refer to $E$ as the \emph{ground set} of the clutter. We call members of $\mathcal{A}$ \emph{rows} of the clutter. In the literature, elements of the ground set are often referred to as vertices, while rows are called edges. Since we will later represent rows of a clutter by vertices in a graph, we prefer to avoid this terminology. If $M$ is a clutter, then $E(M)$ denotes its ground set.

For an example of a clutter, we may take the rows to be the circuits of a matroid. Thus clutters are natural generalisations of matroids: they lie somewhere on the spectrum between matroids, and completely general hypergraphs. It may seem as though clutters are significantly more general objects than matroids, but there are some reasons to view them as being closer to the matroid end of the spectrum. In particular, there are notions of deletion and contraction for clutters. If $M=(E,\mathcal{A})$ is a clutter, and $v$ is an element of $E$, we define $M\backslash v$ to be the clutter
\[
(E-v,\{A\in\mathcal{A}\colon v\notin A\})
\]
and we let $M/v$ be the clutter on the set $E-v$ whose rows are the sets in $\{A-v\colon A\in \mathcal{A}\}$ that are minimal under subset-inclusion. We say that $M\backslash v$ and $M/v$ are produced by \emph{deleting} and \emph{contracting} $v$ respectively. These clutter operations extend the matroidal operations: if $M$ is the clutter of circuits in the matroid $N$, then $M\backslash v$ and $M/v$ are the clutters of circuits in the matroids $N\backslash v$ and $N/v$. Any clutter produced from $M$ by a (possibly empty) sequence of deletions and contractions is a \emph{minor} of $M$. A minor produced by a non-empty sequence of deletions and contractions is a \emph{proper minor}.

The following result is in \cite{cornuejols}, and shows that the order of deletion and contraction is immaterial.

\begin{proposition}
Let $M=(E,\mathcal{A})$ be a clutter, and let $v$ and $v'$ be elements of $E$. Then (i) $(M\ba v)\ba v'=(M\ba v')\ba v$, (ii) $(M/ v)/v'=(M/v')/ v$, and (iii) $(M\ba v)/v'=(M/v')\ba v$.
\end{proposition}

Clutters, moreover, have a duality involution that is analogous to matroid duality. If $M=(E,\mathcal{A})$ is a clutter, then the \emph{blocker} of $M$, written $b(M)$, has $E$ as its ground set, and its rows are the minimal subsets of $E$ that have non-empty intersection with each row of $M$. Edmonds and Fulkerson \cite{eandf} proved that $b(b(M))=M$. This involution swaps deletion and contraction, just as matroid duality does. Thus $b(M\ba v)=b(M)/v$ and $b(M/v)=b(M)\ba v$.

In this article we present evidence that pushes clutters further in the matroid direction along the matroid-hypergraph continuum. We show that some connectivity behaviour in matroids is actually just a special case of a clutter phenomenon. To do so, we must develop a notion of connectivity for clutters.

\begin{definition}
Let $M=(E,\mathcal{A})$ be a clutter. A \emph{separation} of $M$ is a partition of $E$ into non-empty parts, $X$ and $Y$, such that every row is contained in $X$ or $Y$. If $M$ admits no separation then it is \emph{connected}.
\end{definition}

This is a natural way to define connectivity for clutters, since it generalises connectivity for graphs and for matroids. If $M=(E,\mathcal{A})$ is a clutter and each row has cardinality two, then $M$ can be identified with a simple graph $G$, with vertex set $E$, whose edges are the rows of $M$. In this case, $M$ is connected if and only if $G$ is. Similarly, if the rows of $M$ are the circuits of a matroid, $N$, then separations of $M$ and $N$ exactly coincide. Therefore $M$ is connected if and only if $N$ is.

We would like to know which inductive properties of matroid connectivity extend to connected clutters. Our first observation is a negative one. If $N$ is a connected matroid, and $e$ is an element of its ground set, then either $N\backslash e$ or $N/e$ is a connected matroid \cite{oxley}[Theorem 4.3.1]. This phenomenon does not extend to clutters. To see this, consider a clutter, $M$, whose edges all have cardinality two, and therefore correspond to the edges of a graph, $G$. Assume $v$ is a cut-vertex in $G$. Then $M\backslash v$ corresponds to the graph produced from $G$ by deleting $v$ and all edges incident with it. This is certainly not a connected clutter. On the other hand, $M/v$ is produced by removing $v$, all rows containing $v$, all rows containing a neighbour of $v$ in $G$, and then adding all such neighbours as singleton rows. It is clear that this clutter will also fail to be connected.

On the other hand, our main theorem is positive. Brylawski \cite{brylawski} and Seymour \cite{seymour} independently proved that if $M$ is a connected matroid with a connected proper minor, $N$, then we can delete or contract an element from $M$ in such a way to preserve connectivity, and the minor $N$. We prove that this is a special case of a clutter phenomenon.

\begin{theorem}
\label{maintheorem}
Let $M$ and $N$ be connected clutters and assume that $N$ is a proper minor of $M$. There exists an element, $v\in E(M)$, such that either $M\backslash v$ or $M/v$ is connected and has $N$ as a minor.
\end{theorem}

This type of theorem is known as a \emph{splitter theorem}, after Seymour's well-known splitter theorem for $3$-connected matroids \cite{seymoursplitter}. We obtain, as a corollary, a weaker type of statement, known as a \emph{chain theorem}.

\begin{corollary}
Let $M$ be a non-empty connected clutter. Then there is an element, $v\in E(M)$, such that either $M\backslash v$ or $M/v$ is a connected clutter.
\end{corollary}

\begin{proof}
Since every clutter has the empty clutter as a minor, we simply apply Theorem~\ref{maintheorem} with $N$ equal to the empty clutter.
\end{proof}

We note that our notion of connectivity is not invariant under taking blockers. To see this, let $M$ be a clutter whose rows are the circuits of a matroid, $N$. Assume that $N$ admits a separation, $(X,Y)$, but that the dual matroid, $N^{*}$, has no circuits of size less than three. By an earlier observation, $(X,Y)$ is also a separation of $M$. The rows of $b(M)$ are the bases of $N^{*}$ \cite{fulkerson2}. Assume that $(X',Y')$ is a separation of $b(M)$, and let $x$ and $y$ be elements from $X'$ and $Y'$, respectively. Then no basis of $N^{*}$ contains both $x$ and $y$, so $N^{*}$ contains a circuit of size at most two, contrary to hypothesis. Thus $b(M)$ is a connected clutter, even though $M$ is not.

The main tool we use to prove Theorem~\ref{maintheorem} is the \emph{incidence graph} of a clutter. Let $M=(E,\mathcal{A})$ be a clutter. We use $G(M)$ to denote the incidence graph of $M$. The vertex set of $G(M)$ is $E\cup \mathcal{A}$. We say that vertices in $E$ are \emph{black} and vertices in $\mathcal{A}$ are \emph{white}. Every edge of $G(M)$ joins a black vertex to a white vertex, so $G(M)$ is bipartite. The vertex $v\in E$ is adjacent to $A\in\mathcal{A}$ in $G(M)$ if and only if $v$ is contained in $A$.

The incidence graph allows us to study clutter connectivity in graph theoretical terms.

\begin{proposition}
\label{connectediff}
Let $M$ be a clutter. If $G(M)$ is connected, then $M$ is connected. If $M$ is connected, and is not the clutter with a single element and one, empty, row, then $G(M)$ is connected.
\end{proposition}

\begin{proof}
Assume that $G(M)$ is not connected. We will prove that either $M$ is not connected, or $M$ is equal to the special clutter described in the statement of the proposition. Let $(A,B)$ be a partition of the vertices of $G(M)$ into non-empty parts, such that no edge joins a vertex in $A$ to a vertex in $B$. Assume that both $A$ and $B$ contain elements of $E(M)$. Then $(A\cap E(M),B\cap E(M))$ is clearly a separation of $M$, and $M$ is not connected. Therefore we will assume that $A$ contains no element of $E(M)$. Thus every vertex in $A$ is white. It follows that $G(M)$ has a white vertex that is connected to no black vertex, and that therefore $M$ has an empty row. Since $M$ is a clutter, it follows that $M$ has exactly one, empty, row, and that therefore $G(M)$ contains a single white vertex, and no edges. Note that $E(M)$ is non-empty, for otherwise $G(M)$ contains a single vertex, and is therefore connected. If $E(M)$ contains at least two elements, then we can find a separation of $M$. Thus we assume that $E(M)$ contains exactly one element, and deduce that $M$ is the clutter described in the proposition. 

Now suppose $M$ is not connected and has a separation $(A,B)$. White vertices corresponding to rows in $A$ are incident only with elements of $A$; white vertices corresponding to rows in $B$ are incident only with elements of $B$. There are no other white vertices in $G(M)$, so this means that there are no paths between vertices in $A$ and vertices in $B$. Thus $G(M)$ must be disconnected.
\end{proof}

\section{Proof of the main theorem}

If $v$ is a vertex of a graph, then \openN{v} represents the set of neighbours of $v$ (this set excludes $v$). We say \openN{v} is the \emph{open neighbourhood} of $v$. We write \closedN{v} for the \emph{closed neighbourhood} of $v$. That is, $\closedN{v}=\openN{v}\cup\{v\}$. In order for a bipartite graph with black and white vertices to be the incidence graph of a clutter, if $u$ and $v$ are distinct white vertices of $G$, then $\openN{u}$ cannot be a subset of $\openN{v}$.

The next result follows immediately from the definition of deletion in clutters.

\begin{proposition}
If $M$ is a clutter, and $v$ is in $E(M)$, then $G(M\backslash v)=G(M)\backslash \closedN{v}$.
\end{proposition}

Clutter contraction is somewhat more complicated to observe in the incidence graph. We will use only one special case of contraction. We say that the black vertices, $u$ and $v$, are \emph{twins} if $\openN{u}=\openN{v}$.

\begin{proposition}
\label{contracttwin}
Let $M$ be a connected clutter. If $v$ and $v'$ are twin black vertices, then $G(M/v)=G(M)\ba v$ and is therefore connected.
\end{proposition}

\begin{proof}
We form $M/ v$ by removing the element $v$ from $E(M)$ and taking the rows, with $v$ deleted, which are minimal under subset-inclusion. Suppose that $G(M)\ba v$ has two distinct white vertices, $u$ and $w$, such that every neighbour of $u$ is a neighbour of $w$. This property does not hold in $G(M)$, so $v$ must have been adjacent to $u$ but not $w$. As $v'$ is a twin of $v$, then $v'$ is also adjacent to $u$ and not $w$. These adjacencies remain in $G(M)\ba v$, and so we have $\openN{u}\nsubseteq\openN{z}$. This shows that in $G(M)\ba v$, there is no pair of distinct white vertices, one of whose neighbourhood is contained in the other. It follows that $G(M)\ba v$ is the incidence graph of $M/v$. 

Finally, it is clear that $G(M)\ba v$ is connected as for every path using $v$, replacing $v$ with $v'$ gives a second path, and so deleting $v$ cannot increase the number of components in the graph.
\end{proof}

With this setup, we can immediately begin the proof of the main result.

\begin{theorem}
Let $M$ and $N$ be connected clutters and assume that $N$ is a proper minor of $M$. There exists an element, $v\in E(M)$, such that either $M\backslash v$ or $M/v$ is connected and has $N$ as a minor.
\end{theorem}

\begin{proof}
Assume that $M$ and $N$ form a counterexample to the theorem. We will let $G$ stand for $G(M)$.

\begin{subproposition}
\label{subprop1}
There is an element $v\in E(M)$ such that $N$ is a minor of $M\ba v$.
\end{subproposition}

\begin{subproof}
Assume that this is not the case. Then $N$ is a minor of $M/u$ for some $u\in E(M)$. Now $M/u$ is not connected, or else $M$ and $N$ would not give us a counterexample. Therefore $G(M/u)$ is disconnected by Proposition \ref{connectediff}. Since $N$ is connected, it follows easily that $E(N)$ is contained in a connected component of $G(M/u)$. Choose $C$, a component of $G(M/u)$ such that $C$ does not contain $E(N)$. If $C$ consists of a single white vertex, then $M/u$ has an empty row, and this means that it has exactly one row. Hence $G(M/u)$ contains a single white vertex and no edges. This means that $N$ contains at most one element, or else it is not connected. If $M/u=N$, then $M/u$ is connected, and so $E(M/u)$ must contain an element that is not in $E(N)$. Let $u'$ be such an element. Then $u'$ is an isolated black vertex in $G(M/u)$, so $N$ is a minor of $M/u\ba u'$ and hence of $M\ba u'$, contrary to assumption. We must now assume that $C$ contains a black vertex, $u'$. As $C$ is a component of $G(M/u)$ and $C$ does not contain any element of $E(N)$, we see that $N$ is a minor of $M/u\ba u'$ and hence of $M\ba u'$, which is a contradiction.
\end{subproof}

Note that if $u$ and $v$ are black vertices, then $\openN{u}$ may be a subset of $\openN{v}$. Say that $v$ is a \emph{minimal} black vertex if there is no black vertex, $u$, such that $\openN{v}$ properly contains $\openN{u}$.

\begin{subproposition}
There is a minimal black vertex, $v$, of $G$, such that $N$ is a minor of $M\ba v$.
\end{subproposition}

\begin{subproof}
Assume the statement is false. By \ref{subprop1}, we can choose a non-minimal black vertex $v'$ so that $N$ is a minor of $M\backslash v'$. Let $v'$ have the smallest possible degree. First assume $|E(N)|>1$, so $G(N)$ is connected by Proposition \ref{connectediff}. Say $G\backslash\closedN{v'}$ has components $C_1,\ldots,C_t$, where $E(N)\subseteq C_1$. As $v'$ is not minimal, we will choose a minimal black vertex $v$ with $\openN{v}\subset\openN{v'}$. Note this implies $v$ is an isolated vertex in $G\backslash\closedN{v'}$, and so is one of the components $C_1,\ldots,C_t$. Then $v$ is clearly not in $C_1$, so $N$ is a minor of $M\backslash v'\backslash v$, and hence of $M\backslash v$, as desired.

\begin{figure}[ht]
\centering
\includegraphics[scale=0.75,trim={7cm 16cm 7cm 4cm},clip]{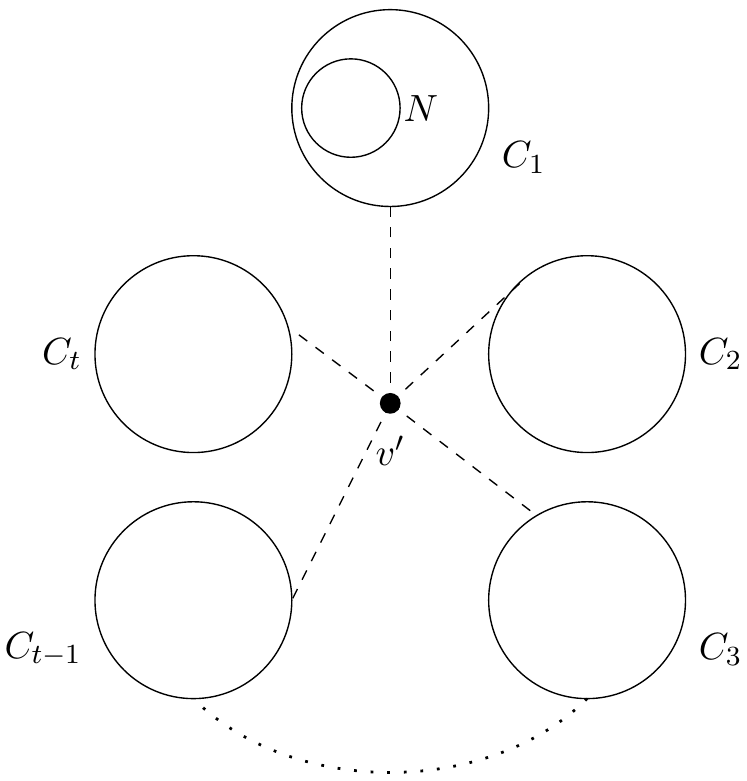}
\end{figure}

Now we consider the case that $|E(N)|$ is at most $1$. We will still assume $v'$ is not minimal, so $\openN{v}$ is properly contained in $\openN{v'}$, for some minimal black vertex $v$. If $v\notin C_1$, then $M$ is a minor of $M\backslash v'\backslash v$, and hence of $M\backslash v$, for the same reasons as in the previous case. Thus $v\in C_1$, implying $C_1$ is a single vertex as $v$ is isolated in $G\ba\closedN{v'}$. If $|E(N)|=0$, then $v$ can be any minimal black vertex and the result follows as $N$ will be a minor of $M\ba v$. Thus we assume that $E(N)=\{v\}$. Note that $v$ is isolated after deleting $\closedN{v'}$. This means $N$ is the clutter with $E(N)=\{v\}$ and no rows. Choose a black vertex $u'\in C_2$. Then $N$ is a minor of $M\ba u'$. If $u'$ is a minimal black vertex, the result follows. So let $u$ be a minimal black vertex with $\openN{u}\subset\openN{u'}$. If $u\in C_2$, then $N$ is a minor of $M\ba u$ and the result also follows. If $u\notin C_2$, then no  neighbour of $u$ is in $C_{2}$, but any such neighbour is also a neighbour of $u'$. It follows that all the neighbours of $u$ are also neighbours of $v'$. This means $u$ is an isolated vertex after deleting $\closedN{v'}$, so $N$ is a minor of $M\ba u$.
\end{subproof}


Now fix $v$ to be a minimal black vertex such that $N$ is a minor of $M\ba v$. Let $C_1,\ldots,C_t$ be the connected components of $G(M\ba v)$, where $t\geq 2$. Since $N$ is connected, we can assume that $E(N)$ is contained in $C_1$.

\begin{subproposition}
\label{notwin}
If $u$ is a black vertex which is not in $C_1$, then $u$ has no twin vertex.
\end{subproposition}

\begin{subproof}
First suppose that $u=v$, and let $u'$ be a twin of $u$. We know that $N$ is a minor of $M\ba u$, and we have that $u'$ is isolated in $G(M\ba u)$. It follows easily from Propositions \ref{connectediff} and \ref{contracttwin} that $M/u'$ is connected. Moreover, as $G(M/u')=G(M)\ba u'$, the only way $M/u'$ cannot contain $N$ as a minor is if $E(N)=\{u'\}$. But this would imply that $M/v$ does contain $N$ as a minor, and is connected, and the theorem would follow. Now let $u\neq v$. Assume that $u$ is a black vertex in $C_i$ where $i\neq 1$, and assume that $u'$ is a twin of $u$. Since $u$ is contained in $C_i$, a component of $G(M\ba v)$, and $E(N)$ is contained in $C_1$, it follows that $N$ is a minor of $M\ba v/u$, and hence of $M/u$. But Propositions \ref{connectediff} and \ref{contracttwin} imply that $M/u$ is connected, a contradiction to our counterexample.
\end{subproof}

If $u$ is any minimal black vertex, and $C$ is any connected component of $G\backslash\closedN{u}$, then we define $C$ to be a \emph{good component}. Therefore $C_1,\ldots,C_t$ are good components. The next claim shows that good components contain minimal black vertices.

\begin{subproposition}
\label{minimal}
Let $u$ be a minimal black vertex, and assume that $u$ is not in $C_{1}$. Let $C$ be a component of $G\ba\closedN{u}$.
Then $C$ contains a minimal black vertex.
\end{subproposition}

\begin{subproof}
 First, we prove that $C$ contains a black vertex. If not, then $C$ is a single white vertex, $w$. Since $G(M)$ is connected, $w$ is adjacent with a black vertex, $x$, in $G(M)$. Neither $w$ nor $x$ belongs to $\closedN{u}$, so $w$ and $x$ and adjacent in $G\ba \closedN{u}$. Thus $C$ contains the black vertex, $x$, contrary to hypothesis. Therefore $C$ contains at least one black vertex. Let $x'$ be an arbitrary black vertex in $C$. If $x'$ is minimal, the result follows. Hence assume that there is a black vertex $x$ with $\openN{x}\subset\openN{x'}$. Choose $x$ so that its degree is as small as possible, implying that $x$ is minimal. If $x\in C$, the result follows, so say $x\in D$ where $D$ is some connected component of $G\ba\closedN{u}$ other than $C$. If $x$ is adjacent to a white vertex in $D$, we will clearly not have $\openN{x}\subset\openN{x'}$. So $D=\{x\}$, and $x$ is adjacent only to neighbours of $u$. As $u$ is a minimal black vertex, we deduce that $x$ and $u$ are twin vertices, which contradicts \ref{notwin}.
\end{subproof}

 Note that it is possible for one good component to be contained in another. Let $u$ be a minimal black vertex. We say that a component $C$ in $G\backslash\closedN{u}$ is \emph{minimal} if the vertex set of $C$ does not properly contain the vertex set of a good component.

\begin{subproposition}
\label{2point2}
Let $u$ be a minimal black vertex, and let $C$ be a component in $G\backslash\closedN{u}$ that is disjoint from $C_{1}$. Assume that $C$ is a minimal good component. If $w$ is a black vertex in $C$, then $w$ has a common neighbour with $u$.
\end{subproposition}

\begin{subproof}
Assume there is a black vertex in $C$ that has no common neighbour with $u$. We will prove that there is a minimal black vertex with this property. Let $w'$ be an arbitrary black vertex in $C$ that has no white neighbour in common with $u$. If $w'$ is not a minimal black vertex, then we can assume that $w$ is a minimal black vertex and that $\openN{w}\subset \openN{w'}$. Then $w$ is joined to $w'$ by a path of length $2$, and this path does not contain a white vertex adjacent with $u$, since $w'$ does not have a neighbour in common with $u$. This means that $w$ and $w'$ are joined by a path in $G\backslash\closedN{u}$, so both are in $C$. In fact, $w$ cannot have a neighbour in common with $u$, because any such neighbour would also be a neighbour of $w'$. Thus $w$ is a minimal black vertex in $C$ having no common neighbours with $u$.

Any white vertex not in $C$ that is adjacent to a black vertex in $C$ must also be adjacent to $u$. It immediately follows that any white vertex not in $C$ is not adjacent to $w$. Therefore every vertex not in $C$ is also a vertex in $G\backslash\closedN{w}$. This implies that the vertices not in $C$ are contained in a connected component of $G\backslash\closedN{w}$. Since $w\in C$, and $C$ is disjoint from $C_{1}$, it follows that $N$ is a minor of $M\ba w$. Therefore $M\ba w$ is not connected, since $M$ and $N$ form a counterexample to the theorem. It follows that $G\ba \closedN{w}$ is not connected. Let $D$ be a connected component in $G\backslash\closedN{u}$ that is different from the one containing the vertices not in $C$. Every vertex in $D$ is also in $C$. But the vertex set of $D$ is a proper subset of the vertex set of $C$, since it doesn't contain $w$. Recall that $w$ is a minimal black vertex, and thus $D$ is a good component. This contradicts the minimality of $C$.
\end{subproof}

\begin{subproposition}
\label{alsonotsingle}
Let $u$ be a minimal black vertex, where $u\notin C_1$, and let $C$ be a component in $G\backslash\closedN{u}$. Assume that $C$ is a minimal good component. Then $C$ contains at least two black vertices.
\end{subproposition}

\begin{subproof}
By \ref{minimal}, we see that $C$ contains a minimal black vertex, $w$. Assume that $w$ is the only black vertex in $C$. We also assume that $C$ contains a white vertex, $x$. Then $w$ is the only neighbour of $x$. By \ref{2point2}, we can let $y$ be a common neighbour of $w$ and $u$. Then $\{w\}=\openN{x}\subseteq\openN{y}$, a contradiction, since $G$ is the incidence graph of a clutter. Therefore $C=\{w\}$. Thus $\openN{w}\subseteq\openN{u}$. Since $u$ is a minimal black vertex, this means that $\openN{w}=\openN{u}$, so $u$ has a twin vertex. As $u\notin C_1$, this is a contradiction to \ref{notwin}.
\end{subproof}



 \begin{subproposition}
 \label{nbhd}
 Let $u$ be a minimal black vertex that is not in $C_{1}$. Let $C$ be a component of $G\ba \closedN{u}$, and assume that $C$ is a minimal good component. If $w$ is a minimal black vertex in $C$, then the component of $G\ba \closedN{w}$ that contains $u$ also contains every vertex in $C\ba \closedN{w}$.
\end{subproposition}

\begin{subproof}
Note that $C$ contains at least two black vertices by \ref{alsonotsingle}. Therefore $C\ba \closedN{w}$ contains at least one vertex. Let $x$ be an arbitrary vertex in $C\ba \closedN{w}$, and let $C'$ be the component of $G\ba \closedN{w}$ containing $x$. There must be a vertex in $C'$ that is not in $C$, for otherwise the vertex set of $C'$ is a proper subset of the vertex set of $C$, which contradicts the minimality of $C$. It follows that in $G\ba\closedN{w}$, there is a path from $x$ to a vertex not in $C$. Any such path must contain a neighbour of $u$. It now follows that in $G\ba\closedN{w}$ there is a path from $x$ to $u$. As $x$ was chosen arbitrarily from $C\ba \closedN{w}$, we see that the component of $G\ba\closedN{w}$ that contains $u$ also contains every vertex in $C\ba \closedN{w}$, exactly as desired.
\end{subproof}

\begin{subproposition}
At least one of the components $C_2,\ldots,C_t$ is not minimal.
\end{subproposition}

\begin{subproof}
Assume that $C_2,\ldots, C_t$ are all minimal good components. Say a black vertex, $u$, in one of $C_2,\ldots,C_t$ is \emph{interesting} if $v$ is in the same component as $C_1$ in $G\backslash\closedN{u}$. Assume $u$ is interesting, and chosen so that $|\openN{u}\cap\openN{v}|$ is smallest possible. We assume that $u$ is in $C_{i}$, where $i\geq 2$. Note that $N$ is a minor of $M\ba u$, and hence $M\ba u$ is not connected. Thus $G\ba \closedN{u}$ is not connected. Let $D$ be a component in $G\backslash\closedN{u}$ not containing $v$. Then $D$ has no vertex in common with $C_{i}\ba\closedN{u}$, by \ref{nbhd}. It also has no vertex in common with $C_{1}$, since $v$ is in the same component as $C_{1}$ in $G\ba \closedN{u}$. Therefore any vertex that is not in $D$, but is adjacent to a vertex in $D$, must be a neighbour of $w$ that is not in $C_{i}$. Any such vertex is also a neighbour of $v$. Hence $D$ is a connected component of $G\ba \closedN{v}$. Now we can assume that $D=C_j$, where $i\ne j$ and $j\geq 2$. Choose $w$, a black vertex in $C_j$. In order for $C_j$ to be disconnected from the component containing $v$ in $G\ba\closedN{u}$, we must have that $\openN{w}\cap\openN{v}\subseteq\openN{u}\cap\openN{v}$, implying that $w$ is interesting. Now $\openN{w}\cap\openN{v}=\openN{u}\cap\openN{v}$ by the choice of $u$. Since $w$ was arbitrary, this means $C_j\ba\closedN{w}$ is disconnected from $v$ in $G\backslash\closedN{w}$, contradicting \ref{nbhd}.

Now assume that there is no interesting vertex. Let $y$ be a vertex in $\openN{v}$ that is adjacent to a vertex in $C_1$. Let $w$ be an arbitrary black vertex in $C_2$. Then $C_1$ is disconnected from $v$ in $G\backslash\closedN{w}$, as $w$ is not interesting by assumption, which implies $y\in\openN{w}$. As $w$ was arbitrary, $y$ is adjacent to every black vertex in $C_2$, as well as $v$. Thus if $x$ is a white vertex in $C_2$, then $\openN{x}\subset\openN{y}$, a contradiction as $G$ is the incidence graph of a clutter.
\end{subproof}

From now on, we assume that $C_2$ is a good component, but not minimal. Let $C$ be a minimal good component, and assume that the vertex set of $C$ is properly contained in the vertex set of $C_2$. Let $u$ be a minimal black vertex such that $C$ is a component of $G\ba\closedN{u}$.

\begin{subproposition}
$u\in C_2$.
\end{subproposition}

\begin{subproof}
If this is not the case, then any common neighbour of $u$ and a vertex in $C$ must be a common neighbour of $v$ and a vertex in $C_2$. This implies that $C$ is a connected component of $G\ba \closedN{v}$, which is impossible because the veretx set of $C$ is properly contained in the vertex set of a connected component of $G\ba \closedN{v}$.
\end{subproof}

By \ref{minimal}, we can choose a minimal black vertex, $w$, in $C$. Let $H$ be the component of $G\ba \closedN{w}$ that contains $u$. By \ref{nbhd}, $H$ also contains $C\ba\closedN{w}$. Assume that we have chosen $C$, $u$, and $w$, so that $H$ is as large as possible.

\begin{subproposition}
\label{twopointsix}
Let $D$ be a component of $G\ba \closedN{w}$ not equal to $H$. Then $D$ is a minimal good component.
\end{subproposition}

\begin{subproof}
Note that $D$ is a good component, since it is disconnected when we delete $w$ and its neighbours. Assume $D$ is not minimal. Let $D'$ be a minimal good component such that the vertex set of $D'$ is properly contained in the vertex set of $D$. Let $u'$ be a minimal black vertex such that $D'$ is a component in $G\backslash\closedN{u'}$. Suppose $u'\notin D$. Then any neighbour of $u'$ that is adjacent to a vertex in $D'$ is not in $D$, but is adjacent to a vertex in $D$. The only such vertices are in $\closedN{w}$. This means that $D'$ is a component of $G\ba \closedN{w}$, but this is impossible, since the vertex set of $D'$ is properly contained in a the vertex set of a component of $G\ba \closedN{w}$. Therefore $u'$ is in $D$.

\begin{figure}[ht]
\centering
\includegraphics[scale=0.75,trim={6cm 17.5cm 6.5cm 4cm},clip]{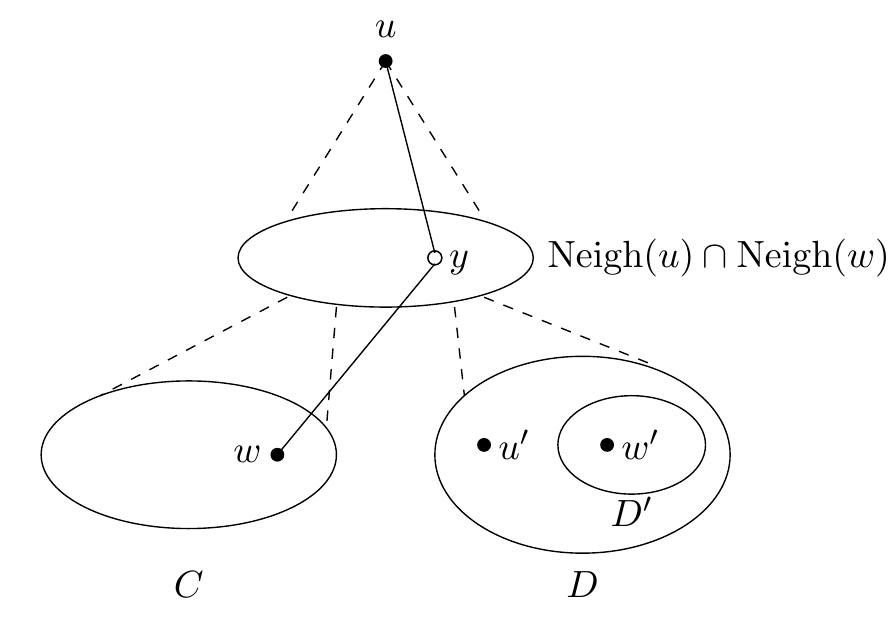}
\end{figure}

Assume that there is vertex, $y$, in $\openN{w}\cap\openN{u}$ that is adjacent to $u'$. Let $w'$ be an arbitrary minimal black vertex contained in $D'$, and assume that $y$ is not a neighbour of $w'$. Let $x$ be an arbitrary vertex in $H$. Then $x$ and $u$ are joined by a path, $P_{x}$, in $G\ba\closedN{w}$. By concatenating $P_{x}$ with the two edges $uy$ and $yu'$, we obtain a path joining $x$ to $u'$. Assume that this is not a path in $G\ba\closedN{w'}$, so that some vertex in the path is adjacent to $w'$. Such a vertex can only be a white vertex, so it is not $u$ or $u'$. Moreover, we have assumed that $y$ is not adjacent to $w'$. Therefore some vertex in $P_{x}$ is adjacent to $w'$. But $P_{x}$ is a path in $G\ba\closedN{w}$ that contains $u$, and $u$ is not in $D$, a connected component of $G\ba \closedN{w}$. Therefore there can be no edge from a vertex in $D$, such as $w'$, to a vertex in $P_{x}$. Now we see that the component of $G\ba\closedN{w'}$ that contains $u'$ also contains $x$. As $x$ was arbitrary, this component contains every vertex in $H$, as well as $u'$. This violates our choice of $C$, $u$, and $w$, since we could have chosen $D'$, $u'$, and $w'$ instead. We conclude that $y$ is adjacent to $w'$. Since $w'$ was an arbitrary minimal black vertex in $D'$, we conclude that every minimal black vertex in $D'$ is adjacent to $y$. Next we will show that every black vertex in $D'$ is adjacent to $y$.

Let $x'$ be a black vertex in $D'$. If $x'$ is minimal, then we are done, so assume otherwise. Then there is a black vertex, $x$, such that $\openN{x}\subset \openN{x'}$. We may as well assume that $x$ is a minimal black vertex. If $x$ is in $D'$, then $x$ is adjacent to $y$, so $x'$ is adjacent to $y$, as desired. Therefore we assume that $x$ is not in $D'$. Then $\openN{x}\subseteq \openN{u'}$. As $u'$ is a minimal black vertex, we deduce that $\openN{x}= \openN{u'}$. Since $u'$ is not contained in $C_{1}$, it cannot be the case that $u'$ has a twin vertex, by \ref{notwin}. Therefore $x$ and $u'$ are the same vertex. But $y$ is adjacent to $u'$, and now we again conclude that $x$, and hence $x'$, is adjacent to $y$, as desired. Therefore every black vertex in $D'$ is adjacent to $y$. This means that if $z$ is an arbitrary white vertex in $D'$, then every neighbour of $z$ is a neighbour of $y$, so $\openN{z}\subseteq\openN{y}$ which is a contradiction to the fact that $G$ is the incidence graph of a clutter. We must conclude that $u'$ is not adjacent to any vertex in $\openN{w}\cap\openN{u}$. This also means that no black vertex in $D'$ can be adjacent to a vertex in $\openN{w}\cap\openN{u}$, since any vertex that is not in $D'$, but is adjacent to a black vertex in $D'$, is adjacent to $u'$.

Let $P$ be a shortest-possible path between $u$ and $u'$ in $G$. First assume that there is no vertex in $P$ that is adjacent to a vertex in $D'$. Let $w'$ be an arbitrary minimal black vertex in $D'$. Then $P$ is a path from $u$ to $u'$ in $G\ba\closedN{w'}$. If $x$ is an arbitrary vertex in $H$, then $x$ is joined by a path, $P_{x}$, to $u$ in $G\ba\closedN{w}$. Since $w'$ is not adjacent to any vertex in $\openN{u}\cap\openN{w}$, we see that $P_{x}$ is also a path in $G\ba\closedN{w'}$. By concatenating $P_{x}$ and $P$, we obtain a path from $x$ to $u'$ in $G\ba \closedN{w'}$.  Therefore the component of $G\ba\closedN{w'}$ that contains $u'$ also contains every vertex in $H$, and we again have a contradiction to our choice of $C$, $u$, and $w$. Thus there is a vertex in $P$ that is a neighbour of a vertex in $D'$.

Note that any vertex not in $D'$ that is a neighbour of a vertex in $D'$ is in $\openN{u'}$, as $D'$ is a connected component of $G\ba\closedN{u'}$. Since $P$ is a shortest path from $u$ to $u'$, we see that $P$ contains exactly one vertex, $y$, that is adjacent to a vertex in $D'$. Let $w'$ be an arbitrary minimal black vertex in $D'$. If $y$ is not adjacent to $w'$, then $P$ is a path from $u$ to $u'$ in $G\ba\closedN{w'}$. We get a contradiction to our choice of $C$, $u$, and $w$, exactly as before. Therefore $y$ is adjacent to every minimal black vertex in $D'$.

We show that $y$ is adjacent to every black vertex in $D'$. Let $x'$ be a black vertex in $D'$, and assume that $x'$ is not adjacent to $y$. Then $x'$ is not a minimal black vertex, so let $x$ be a minimal black vertex such that $\openN{x}\subset\openN{x'}$. If $x$ is in $D'$, then $y\in\openN{x}$, and we have a contradiction, so $x\notin D'$. This means that $\openN{x}\subseteq \openN{u'}$. Because $u'$ is a minimal black vertex, and does not have a twin by \ref{notwin}, this implies that $x=u'$. But $y$ is in $\openN{u'}$, so we again see that $y$ is adjacent to $x$. Thus $y$ is adjacent to every black vertex in $D'$. If $z$ is a white vertex in $D'$, then every neighbour of $z$ is a neighbour of $y$, which is impossible. From this final contradiction we see that $D$ must be a minimal good component.
\end{subproof}


Now we can finish the proof of the main theorem. Let $D$ be a component of $G\ba \closedN{w}$ that is not equal to $H$. By \ref{twopointsix}, we see that $D$ is a minimal good component. Any vertex not in $D$, but adjacent to a vertex in $D$, must be in $\openN{u}\cap\openN{w}$. It therefore follows that $D$ is a component of $G\ba \closedN{u}$. By \ref{minimal} we see that $D$ contains a minimal black vertex, $w'$. Let $x$ be an arbitrary vertex in $H$, and let $P_{x}$ be a path from $x$ to $u$ in $G\ba\closedN{w}$. Assume that $P_{x}$ is not a path in $G\ba \closedN{w'}$. Then some vertex of $P_{x}$ is adjacent to $w'$. No vertex in $P_{x}$ is in $D$, since $P_{x}$ is a path in $G\ba\closedN{w}$ containing $u$, and $D$ is a component of $G\ba\closedN{w}$ that does not contain $u$. Therefore there is a vertex in $P_{x}$ that is not in $D$, but is adjacent to a vertex in $D$ (namely $w'$). Any such vertex must be in $\openN{u}\cap\openN{w}$. But this is impossible, because no vertex of $P_{x}$ is in $\closedN{w}$. Let $H'$ be the component of $G\ba\closedN{w'}$ that contains $u$. We have just shown that $H'$ contains all the vertices of $H$. By \ref{nbhd}, we see that $H'$ also contains $D-\closedN{w'}$, and \ref{alsonotsingle} implies that this set is not empty. Thus we have contradicted our choice of $C$, $u$, and $w$, because we could have chosen $D$, $u$, and $w'$ instead. Thus there is no possible counterexample, and the result follows.
\end{proof}

This research did not receive any specific grant from funding agencies in the public, commercial, or
not-for-profit sectors.

\bibliographystyle{acm}
\bibliography{library} 

\end{document}